\documentclass[10pt]{article}
\usepackage{graphicx}
\usepackage{amsthm}
\usepackage{amsmath}
\usepackage{amssymb}
\usepackage{amsmath}
\usepackage{amssymb}

\newtheorem{theorem}{Theorem}
\newtheorem{proposition}{Proposition}[section]
\newtheorem{lemma}[proposition]{Lemma}

\newcommand{\be}{\begin{equation}}
\newcommand{\ee}{\end{equation}}
\newcommand{\nin}{\noindent}
\DeclareMathOperator{\col}{col}
\DeclareMathOperator{\E}{E}
\DeclareMathOperator{\Var}{Var}
\newcommand{\pF}{{\Pr}_{f}}
\newcommand{\pf}{{\Pr}_{p}}
\begin{document}

\title{Distinguishing a truncated random permutation from a random function}

\author{Shoni Gilboa\thanks{Mathematics Department, The Open University of Israel, Raanana 43107, Israel. Email: \texttt{tipshoni@gmail.com}}
\and
Shay Gueron  \thanks{Mathematics Department, University of Haifa, Haifa 31905, Israel} \thanks{ Intel Corporation} 
}

\maketitle

\begin{abstract}
An oracle chooses a function $f$ from the set of $n$ bits strings to itself, which is either a randomly chosen permutation or a randomly chosen function. When queried by an $n$-bit string $w$, the oracle computes $f(w)$, truncates the $m$ last bits, and returns only the first $n-m$ bits of $f(w)$. How many queries does a querying adversary need to submit in order to distinguish the truncated permutation from a random function? 

In 1998, Hall et al.~\cite{Hall} showed an algorithm for determining (with high probability) whether or not $f$ is a permutation, using $O(2^{\frac{m+n}{2}})$ queries. They also showed that if $m < n/7$, a smaller number of queries will not suffice. For $m > n/7$, their method gives a weaker bound.
In this manuscript, we show how a modification of the method used by Hall et al. can solve the porblem completely. It extends the result to essentially every $m$, showing that $\Omega(2^{\frac{m+n}{2}})$ queries are needed to get a non-negligible distinguishing advantage. 
We recently became aware that a better bound for the distinguishing advantage, for every $m<n$, follows from a result of Stam \cite{Stam} published, in a different context, already in 1978. 
\end{abstract}

{\small
\begin{quote}
\end{quote}}

\section{Introduction}
Distinguishing a randomly chosen permutation from a random function is a combinatorial problem which is fundamental in cryptology. A few examples where this problem plays an important role are the security analysis of block ciphers, hash and MAC schemes.
 
One formulation of this problem is the following. An oracle chooses a function $f: \{0, 1\}^n \to \{0, 1\}^n$, which is either a randomly (uniformly) chosen permutation of $\{0, 1\}^n$, or a randomly (uniformly) chosen function from $\{0, 1\}^n$ to $\{0, 1\}^n$. An adversary selects a ``querying and guessing'' algorithm. He first uses it to submit $q$ (adaptive) queries to the oracle, and the oracle responds with $f(w)$ to the query $w \in \{0, 1\}^n$. After collecting the $q$ responses, the adversary uses his algorithm to guess whether or not $f$ is a permutation.  The quality of such an algorithm (in the cryptographic context) is the ability to distinguish between the two cases (rather than successfully guessing which one it is). It is measured by the difference between the probability that the algorithm outputs a certain answer, given that the oracle chose a permutation, and the probability that the algorithm outputs the same answer, given that the oracle chose a function. This difference is called the "advantage" of the algorithm. We are interested in estimating \emph{Adv}, which is the maximal advantage of the adversary, over all possible algorithms, as a function of a budget of $q$ queries. 

The well known answer to this problem is based on the simple ``collision test'' and the Birthday Problem:
$$Adv=1-\left(1-\frac{1}{2^n}\right)\left(1-\frac{2}{2^n}\right)\ldots\left(1-\frac{q-1}{2^n}\right).$$
Since for every $1\leq k\leq q-1$
$$1-\frac{q}{2^n}\leq\left(1-\frac{k}{2^n}\right)\left(1-\frac{q-k}{2^n}\right)\leq\left(1-\frac{q}{2^{n+1}}\right)^2,$$
we get, for $q\leq 2^n$, that
\begin{equation}\label{birthday}1-e^{-\frac{q(q-1)}{2^{n+1}}}\leq 1-\left(1-\frac{q}{2^{n+1}}\right)^{q-1}\leq Adv \leq 1-\left(1-\frac{q}{2^n}\right)^{\frac{q-1}{2}}\leq\frac{q(q-1)}{2^{n+1}}.\end{equation}
This result implies that the number of queries required to distinguish a random permutation from a random function, with success probability significantly larger than, say, $1/2$, is $\Theta(2^{n/2})$.

We now consider the following generalization of this problem: \\

\nin {\bf Problem:}
Let $0 \le m<n$ be integers. An oracle chooses $c \in \{0, 1\}$. If $c = 1$, it picks a permutation $p$ of $\{0, 1\}^n$ uniformly at random, and if $c = 0$, it picks a function $f: \{0, 1\}^n \to \{0, 1\}^n$ uniformly at random. 
An adversary is allowed to submit queries $w \in \{0, 1\}^n$ to the oracle. The oracle computes $\alpha = p(w)$ (if $c=1$) or $\alpha = f(w)$ (if $c=0$), truncates (with no loss of generality) the last $m$ bits from $\alpha$, and replies with the remaining $(n-m)$ bits. The adversary has a budget of $q$ (adaptive) queries, and after exhausting this budget, is expected to guess $c$. 
{\it How many queries does the adversary need in order to gain non-negligible advantage?} \\
Specifically, we seek $q_{1/2}=\min\{q\mid Adv\geq 1/2\}$ as a function of $m$ and $n$. \\

This problem was studied by Hall et al.~\cite{Hall} in 1998. The authors showed (for every $m$) an algorithm that gives a non-negligible distinguishing advantage using $q = O (2^{(n+m)/2})$ queries. They also proved the following upper bound:
\begin{equation}\label{Hall_result}
Adv \le 5\left(\frac{q}{2^{\frac{n+m}{2}}}\right)^{\frac{2}{3}}+\frac{1}{2}\left(\frac{q}{2^{\frac{n+m}{2}}}\right)^3\frac{1}{ 2^{\frac{n-7m}{2}}}.
\end{equation}
For $m\leq n/7$, \eqref{Hall_result} implies that $q_{1/2}=\Omega(2^{\frac{m+n}{2}})$. However, for larger values of $m$, the bound on $q_{1/2}$ that is offered by \eqref{Hall_result} deteriorates, and eventually becomes (already for $m>n/4$) worse than the trivial "Birthday" bound $q_{1/2}=\Omega(2^{n/2})$, which is obviously still valid when the adversary gets only partial (truncated) replies from the oracle.

Hall et al.~\cite{Hall} conjectured that $\Omega(2^{\frac{m+n}{2}})$ queries are needed in order to get a non-negligible advantage, in the general case. 
Surprisingly, it turns out that this was already established $20$ years before the conjecture was raised, in a different context. It follows from the bound
\begin{equation}\label{eq:Stam}Adv\leq\frac{1}{2}\sqrt{\frac{(2^{n-m}-1)q(q-1)}{(2^n-1)(2^n-(q-1)}}\leq\frac{1}{2\sqrt{1-\frac{q-1}{2^n}}}\cdot\frac{q}{2^{\frac{n+m}{2}}},\end{equation}
valid for all $0\leq m<n$, which is a direct consequence of a result of Stam \cite[Theorem 2.3]{Stam} (see also \cite{GGM}). 

In this manuscript we show how the method of proof used in \cite{Hall} can be modified to show the lower bound $q_{1/2}=\Omega(2^{\frac{m+n}{2}})$ for virtually every $m<n$. 
The result follows from explicit upper bounds on \emph{Adv} stated in the following two theorems.

\begin{theorem}
\label{friendly_theorem}
If $m\leq n/3$ then
\begin{equation}\label{thm1}Adv\leq 2\sqrt[3]{2}\left(\frac{q}{2^{\frac{n+m}{2}}}\right)^{2/3}+\frac{2\sqrt{2}}{\sqrt{3}}\left(\frac{q}{2^{\frac{n+m}{2}}}\right)^{3/2}+\left(\frac{q}{2^{\frac{n+m}{2}}}\right)^2.\end{equation}
\end{theorem}

\begin{theorem}\label{friendly_theorem2}
If $n/3<m\leq n-4-\log_2 n$ then
\begin{equation}\label{thm2}Adv\leq 3\left(\frac{q}{2^{\frac{n+m}{2}}}\right)^{2/3}+2\left(\frac{q}{2^{\frac{n+m}{2}}}\right)+5\left(\frac{q}{2^{\frac{n+m}{2}}}\right)^2+\frac{1}{2}\left(\frac{2q}{2^{\frac{n+m}{2}}}\right)^{\frac{n}{n-m}}.\end{equation}
\end{theorem}

The proofs of Theorem \ref{friendly_theorem} and Theorem \ref{friendly_theorem2} are given in Section \ref{sec:main_theorem} and Section \ref{sec:main_theorem2}, respectively.
The proofs follow the same line, but the proof of Theorem \ref{friendly_theorem2} is more elaborate and technical. 

\section{Notation and preliminaries}\label{sec3}

For fixed $m<n$ and $q\leq 2^n$ we denote $\Omega:=\left(\{0,1\}^{n-m}\right)^q$. We view $\Omega$ as the set of all possible sequences of replies that can be given by the oracle to the adversary's $q$ queries.
For $\omega\in\Omega$, let $\pf(\omega)$ and $\pF(\omega)$ be the probabilities that $\omega$ is the actual sequence of replies that the oracle gives to the adversary's $q$ queries, in the case the oracle chose a random permutation or a random function, respectively.
For every $j\geq 2$ and $\omega\in\Omega$, let
$$
\col_j(\omega):=\#\{1\leq i_1<i_2<\ldots<i_j\leq q\mid\omega_{i_1}=\omega_{i_2}=\ldots=\omega_{i_j}\}.
$$

 \begin{lemma}\label{collisions}
For every $j\geq 2$,
\begin{align*}\E_{{f}}\col_j&=\binom{q}{j}\frac{1}{2^{(j-1)(n-m)}},\\
\Var_{{f}}\col_j&\leq\binom{j}{2}\binom{q}{j}\frac{1}{2^{(j-1)(n-m)}}\left(1+\frac{q}{2^{n-m}}\right)^{j-2}.
\end{align*}
\end{lemma}

\begin{proof}
Note that for every $j\geq 2$,
$$\col_j=\sum_{J\in\binom{[q]}{j}}X_J,$$ 
where $\binom{[q]}{j}:=\{\{i_1,i_2,\ldots,i_j\}\mid1\leq i_1<i_2<\ldots<i_j\leq q\}$ and $X_{\{ i_1,i_2,\ldots,i_j\}}$ is the indicator function of the event $\{\omega_{i_1}=\omega_{i_2}=\ldots=\omega_{i_j}\}$. 
Since clearly $\E_{{f}}X_J=(1/2^{n-m})^{|J|-1} $ for every $J$, we immediately get that
$$\E_{{f}}\col_j=\E_{{f}}\sum_{J\in\binom{[q]}{j}}X_J=\sum_{J\in\binom{[q]}{j}}\E_{{f}} X_J=\binom{q}{j}\bigg(\frac{1}{2^{n-m}}\bigg)^{j-1}.
$$
Since $X_{J_1}$ and $X_{J_2}$ are clearly independent whenever $J_1$ and $J_2$ are disjoint, 
\begin{align*}
\Var_{{f}}\col_j=&\Var_{{f}}\sum_{J\in\binom{[q]}{j}}X_J=\sum_{J_1\in\binom{[q]}{j}}\sum_{\substack{{J_2}\in\binom{[q]}{j}\\J_1\bigcap J_2\neq\emptyset}}\E_{{f}} X_{J_1} X_{J_2}-\E_{{f}} X_{J_1} \cdot \E_{{f}} X_{J_2}=\\
=&\sum_{J_1\in\binom{[q]}{j}}\sum_{\substack{{J_2}\in\binom{[q]}{j}\\J_1\bigcap{J_2}\neq\emptyset}}\bigg(\frac{1}{2^{n-m}}\bigg)^{|J_1\bigcup{J_2}|-1}-\bigg(\frac{1}{2^{n-m}}\bigg)^{|J_1|-1}\bigg(\frac{1}{2^{n-m}}\bigg)^{|{J_2}|-1}=\\
=&\binom{q}{j}\frac{1}{2^{(j-1)(n-m)}}\sum_{i=0}^{j-2}\binom{j}{i}\binom{q-j}{i}\left(\frac{1}{2^{i(n-m)}}-\frac{1}{2^{(j-1)(n-m)}}\right)\leq\\
\leq&\binom{j}{2}\binom{q}{j}\frac{1}{2^{(j-1)(n-m)}}\sum_{i=0}^{j-2} \binom{j-2}{i}q^i\frac{1}{2^{i(n-m)}}=\binom{j}{2}\binom{q}{j}\frac{1}{2^{(j-1)(n-m)}}\left(1+\frac{q}{2^{n-m}}\right)^{j-2}.\qedhere
\end{align*}
\end{proof}

The advantage of an algorithm is defined as $\bigl| \pf(E) -\pF(E)  \bigr|$, where $E$ is the event that the algorithm outputs (say) $1$.
The maximum of the advantage of an algorithm, over all possible algorithms, is called the adversary's advantage, and is denoted here by $Adv$.
Clearly,
\begin{equation}\label{eq:Adv}
Adv\leq\max_{E\subseteq \Omega}\left| \pf(E)-\pF(E)\right|=\frac{1}{2}\sum_{\omega\in\Omega}\left| \pf(\omega)-\pF(\omega)\right|,
\end{equation}
with equality, if no computational restricitions are imposed on the adversary. 
We use the following estimate for $Adv$, which is slightly better than a similar bound used in \cite{Hall}.

\begin{lemma}
\label{advantage}
For every $S\subseteq\Omega$,
$$Adv\leq \max_{\omega\in S}\left\lvert\frac{\pf(\omega)}{\pF(\omega)}-1\right\rvert+\pF(\bar{S}).$$
\end{lemma}
\begin{proof} 
Note that 
\begin{multline*}
\sum_{\omega\in\bar{S}}\left| \pf(\omega)-\pF(\omega)\right|\leq\sum_{\omega\in\bar{S}}\left(\pf(\omega)+\pF(\omega)\right)=\\
=\pf(\bar{S})+\pF(\bar{S})=
\pF(S)-\pf(S)+2\pF(\bar{S})=\\
=\sum_{\omega\in S}\left(\pF(\omega)-\pf(\omega)\right)+2\pF(\bar{S})\leq\sum_{\omega\in S}\left|\pf(\omega)-\pF(\omega)\right|+2\pF(\bar{S}).
\end{multline*}
Therefore, using \eqref{eq:Adv}, 
\begin{multline*}
Adv\leq\frac{1}{2}\sum_{\omega\in\Omega}\left| \pf(\omega)-\pF(\omega)\right|=\frac{1}{2}\sum_{\omega\in S}\left| \pf(\omega)-\pF(\omega)\right|+\frac{1}{2}\sum_{\omega\in\bar{S}}\left| \pf(\omega)-\pF(\omega)\right|\leq\\
\leq\sum_{\omega\in S}\left|\pf(\omega)-\pF(\omega)\right|+\pF(\bar{S})=\sum_{\omega\in S}\pF(\omega)\left| \frac{\pf(\omega)}{\pF(\omega)}-1\right|+\pF(\bar{S})\leq\\
\leq\left(\sum_{\omega\in S}\pF(\omega)\right)\max_{\omega\in S}\left\lvert\frac{\pf(\omega)}{\pF(\omega)}-1\right\rvert+\pF(\bar{S})\leq\max_{\omega\in S}\left\lvert\frac{\pf(\omega)}{\pF(\omega)}-1\right\rvert+\pF(\bar{S}).\qedhere\end{multline*}
\end{proof}

In the proofs of Theorem \ref{friendly_theorem} and Theorem \ref{friendly_theorem2} we apply Lemma \ref{advantage} to the set
\begin{equation*}S=\left\{\omega\in\Omega\mid \forall 2\leq j\leq t:\left\lvert \col_j(\omega)-\binom{q}{j}\frac{1}{2^{(j-1)(n-m)}}\right\rvert\leq\alpha_j,\,\col_{t+1}(\omega)\leq\beta\right\},\end{equation*}
where $t\geq 2$ is an integer and $\alpha_1,\alpha_2,\ldots,\alpha_{t-1},\beta$ are positive real numbers, which are chosen apropriately. 
A particular case of this $S$, with  $t=2$, $\alpha_1=c\,q/2^{\frac{n-m+1}{2}}$, $\beta=0$, was used in \cite{Hall}. In this work, we get a refined asymptotic approximation for $Adv$ by using the above general choice of $S$. In the proof of Theorem \ref{friendly_theorem}, we also use $t=2$, but different $\alpha_1$ (which we simply denote $\alpha$) and different $\beta$.

\section{Proof of Theorem \ref{friendly_theorem}}
\label{sec:main_theorem}

The flow of the proof is as follows. 
As mentioned in Section \ref{sec3}, we let
$$S=\left\{\omega\in\Omega:\left\lvert \col_2(\omega)-\binom{q}{2}\frac{1}{2^{n-m}}\right\rvert\leq\alpha\; ,\; \col_3(\omega)\leq\beta\right\},$$
where $\alpha,\beta$ are positive constants to be specifired later.
In Subection \ref{sec:technical} we prove our main technical result, Proposition \ref{proposition}, which provides an upper bound for $\lvert \pf/\pF-1 \rvert$ in $S$.
In Subsection \ref{sec:combine} we first derive, in Lemma \ref{chebyshev}, an upper bound for $\pF(\bar{S})$. Then we combine Lemma \ref{advantage}, Proposition \ref{proposition}, Lemma \ref{chebyshev}, and choose optimal parameters $\alpha, \beta$ to obtain Theorem \ref{friendly_theorem}.

\subsection{Bounding $\lvert \pf/\pF-1 \rvert$ in $S$}

\label{sec:technical}

In this subsection, we prove the following proposition.

\begin{proposition}\label{proposition}
Suppose that $q\leq 2^{n-1}$,
\begin{align}
\label{p1a2}&\frac{\alpha}{2^m}+\frac{2}{3}\cdot\frac{q^3}{2^{2n}}\leq\frac{1}{2},\\
\label{p1a3}&\binom{q}{2}\frac{1}{2^{n-m}}+\alpha\leq\binom{2^{m-1}}{2},\\
\label{p1a4}&\beta\geq 2\binom{q}{3}\frac{1}{2^{2(n-m)}}.
\end{align} 
Then for every $\omega\in S$,
\begin{equation*}\left\lvert\frac{\pf(\omega)}{\pF(\omega)}-1\right\rvert\leq  2\frac{\alpha}{2^m}+2\binom{q}{2}\frac{1}{2^{n+m}}+4\frac{\beta}{2^{2m}}.\end{equation*}
\end{proposition}
In the proof of Proposition \ref{proposition}, we use the following three lemmas.

\begin{lemma}
For every $ x\leq 1$,
\begin{equation}
x\leq -\ln(1-x) \label{eq:ln1},
\end{equation}
for every $0\leq x\leq 1/2$,
\begin{equation}\label{eq:ln3}
-\ln(1-x)-x\leq 2x^2,
\end{equation}
and for every $1\leq x\leq 2$,
\begin{equation}\label{eq:ln2}
x-1\leq 2\ln x.
\end{equation}
\end{lemma}

\begin{proof}
By Taylor expansion, for every $x\leq 1$,
$$-\ln(1-x)=x+\frac{x^2}{2(1-\xi)^2}$$
for some $\xi$ between $0$ and $x$, and \eqref{eq:ln1}, \eqref{eq:ln3} follow. To deduce \eqref{eq:ln2}, note that if $1\leq x\leq 2$, then by \eqref{eq:ln1},
\begin{equation*}
x-1\leq 2\frac{x-1}{x}\leq-2\ln\left(1-\frac{x-1}{x}\right)=2\ln x.
\qedhere\end{equation*}
\end{proof}

\begin{lemma}\label{lemma skt}
Let $s,k$ be positive integers such that $s\leq 2^{k-1}$. Then
$$0\leq -\ln\prod_{i=0}^{s-1}\left(1-\frac{i}{2^k}\right)-\binom{s}{2}\frac{1}{2^k}\leq 4\binom{s}{3}\frac{1}{2^{2k}}+2\binom{s}{2}\frac{1}{2^{2k}}.$$
\end{lemma}

\begin{proof}
For every integer $0\leq i\leq 2^{k-1}$, by \eqref{eq:ln1} and \eqref{eq:ln3},
\begin{equation}\label{ik}0\leq-\ln\left(1-\frac{i}{2^k}\right)-\frac{i}{2^k}\leq 2\left(\frac{i}{2^k}\right)^2=4\binom{i}{2}\frac{1}{2^{2k}}+2\frac{i}{2^{2k}}.\end{equation}
The Lemma follows by summing up the inequalities \eqref{ik} for $0\leq i\leq s-1$, and using the identities $\sum_{i=0}^{s-1}i=\binom{s}{2}$, $ \sum_{i=0}^{s-1}\binom{i}{2}=\binom{s}{3}$.
\end{proof}

\begin{lemma}\label{ln}
Suppose that $q\leq 2^{n-1}$ and let $\omega\in\Omega$ such that
$\col_2(\omega)\leq\binom{2^{m-1}}{2}$. Then,
\begin{equation*}0\leq-\ln\frac{\pf(\omega)}{\pF(\omega)}-\ln\prod_{i=0}^{q-1}\left(1-\frac{i}{2^n}\right)-\frac{\col_2(\omega)}{2^m}\leq 4\frac{\col_3(\omega)}{2^{2m}}+2\frac{\col_2(\omega)}{2^{2m}}.\end{equation*}
\end{lemma}

\begin{proof}
Suppose that in the q-tuple $\omega$, exactly $\ell$ distinct vectors in $\{0,1\}^{n-m}$
appear, with multiplicities $d_1,d_2,\ldots,d_{\ell}$, respectively.
It is easy to verify that $$\pf(\omega)=\frac{\prod_{k=1}^{\ell}\left(\prod_{i=0}^{d_k-1}(2^m-i)\right)}{\prod_{i=0}^{q-1}(2^n-i)},$$
and clearly $\pF(\omega)=(1/2^{n-m})^q$. Therefore, using that $\sum_{k=1}^{\ell}d_k=q$,
$$\frac{\pf(\omega)}{\pF(\omega)}=\prod_{k=1}^{\ell}\left(\prod_{i=0}^{d_k-1}\left(1-\frac{i}{2^m}\right)\right)\cdot\frac{1}{\prod_{i=0}^{q-1}\left(1-\frac{i}{2^n}\right)},$$
hence
\begin{equation}\label{eq:pr_pr}
\ln\frac{\pf(\omega)}{\pF(\omega)}+\ln\prod_{i=0}^{q-1}\left(1-\frac{i}{2^n}\right)=\sum_{k=1}^{\ell}\ln\prod_{i=0}^{d_k-1}\left(1-\frac{i}{2^m}\right).
\end{equation}
For every $1\leq k\leq {\ell}$, note that $d_k\leq 2^{m-1}$, since $\sum_{k=1}^{\ell}\binom{d_k}{2}=\col_2(\omega)\leq\binom{2^{m-1}}{2}$, hence by Lemma \ref{lemma skt}, 
$$0\leq-\ln\prod_{i=0}^{d_k-1}\left(1-\frac{i}{2^m}\right)-\binom{d_k}{2}\frac{1}{2^m}\leq 4\binom{d_k}{3}\frac{1}{2^{2m}}+2\binom{d_k}{2}\frac{1}{2^{2m}}.$$
Summing up on  $1\leq k\leq\ell$, we get that
$$0\leq-\sum_{k=1}^{\ell}\ln\prod_{i=0}^{d_k-1}\left(1-\frac{i}{2^m}\right)-\frac{1}{2^m}\sum_{k=1}^{\ell}\binom{d_k}{2}\leq \frac{4}{2^{2m}}\sum_{k=1}^{\ell}\binom{d_k}{3}+\frac{2}{2^{2m}}\sum_{k=1}^{\ell}\binom{d_k}{2},$$
and the lemma follows by \eqref{eq:pr_pr} since $\sum_{k=1}^{\ell}\binom{d_k}{2}=\col_2(\omega)$ and $\sum_{k=1}^{\ell}\binom{d_k}{3}=\col_3(\omega)$.
\end{proof}

We are now ready to prove Proposition~\ref{proposition}.
\begin{proof}[Proof of Proposition~\ref{proposition}]
For every $\omega\in S$, by \eqref{p1a3} and the definition of $S$,
$$\col_2(\omega)\leq\binom{q}{2}\frac{1}{2^{n-m}}+\alpha\leq\binom{2^{m-1}}{2},$$
hence, by Lemma \ref{ln},
\begin{equation*}0\leq-\ln\frac{\pf(\omega)}{\pF(\omega)}-\ln\prod_{i=0}^{q-1}\left(1-\frac{i}{2^n}\right)-\frac{\col_2(\omega)}{2^m}\leq 4\frac{\col_3(\omega)}{2^{2m}}+2\frac{\col_2(\omega)}{2^{2m}}.\end{equation*}
In addition, by Lemma~\ref{lemma skt} (for $s=q$ and $k=n$), since $q\leq 2^{n-1}$,
\begin{equation*}\label{by5}0\leq-\ln\prod_{i=0}^{q-1}\left(1-\frac{i}{2^n}\right)-\frac{1}{2^n}\binom{q}{2}\leq 4\binom{q}{3}\frac{1}{2^{2n}}+2\binom{q}{2}\frac{1}{2^{2n}}.\end{equation*}
Therefore,
\begin{align}
-\ln\frac{\pf(\omega)}{\pF(\omega)}-\frac{1}{2^m}\left(\col_2(\omega)-\binom{q}{2}\frac{1}{2^{n-m}}\right)&\leq 4\frac{\col_3(\omega)}{2^{2m}}+2\frac{\col_2(\omega)}{2^{2m}},\label{-}\\
\ln\frac{\pf(\omega)}{\pF(\omega)}+\frac{1}{2^m}\left(\col_2(\omega)-\binom{q}{2}\frac{1}{2^{n-m}}\right)&\leq 4\binom{q}{3}\frac{1}{2^{2n}}+2\binom{q}{2}\frac{1}{2^{2n}}.\label{+}
\end{align}
By \eqref{eq:ln1}, \eqref{-} and the definition of $S$,
\begin{align} 
1-\frac{\pf(\omega)}{\pF(\omega)}&\leq-\ln\frac{\pf(\omega)}{\pF(\omega)}\leq\frac{1}{2^m}\left(\col_2(\omega)-\binom{q}{2}\frac{1}{2^{n-m}}\right)+2\frac{\col_2(\omega)}{2^{2m}}+4\frac{\col_3(\omega)}{2^{2m}}\leq\nonumber\\
&\leq \frac{\alpha}{2^m}+\frac{2}{2^{2m}}\left(\binom{q}{2}\frac{1}{2^{n-m}}+\alpha\right)+4\frac{\beta}{2^{2m}}=\left(1+\frac{2}{2^{m}}\right)\frac{\alpha}{2^m}+2\binom{q}{2}\frac{1}{2^{n+m}}+4\frac{\beta}{2^{2m}}\leq\nonumber\\
\label{1-}&\leq 2\frac{\alpha}{2^m}+2\binom{q}{2}\frac{1}{2^{n+m}}+4\frac{\beta}{2^{2m}}.
\end{align}
By \eqref{+} and the definition of $S$,
\begin{align}\ln\frac{\pf(\omega)}{\pF(\omega)}&\leq-\frac{1}{2^m}\left(\col_2(\omega)-\binom{q}{2}\frac{1}{2^{n-m}}\right)+ 4\binom{q}{3}\frac{1}{2^{2n}}+2\binom{q}{2}\frac{1}{2^{2n}}\leq\nonumber\\
\label{denum}&\leq\frac{\alpha}{2^m}+4\binom{q}{3}\frac{1}{2^{2n}}+2\binom{q}{2}\frac{1}{2^{2n}}.\end{align}
In particular, using \eqref{p1a2},
\begin{equation*}\frac{\pf(\omega)}{\pF(\omega)}\leq\exp\left(\frac{\alpha}{2^m}+4\binom{q}{3}\frac{1}{2^{2n}}+2\binom{q}{2}\frac{1}{2^{2n}}\right)<\exp\left(\frac{\alpha}{2^m}+\frac{2}{3}\cdot\frac{q^3}{2^{2n}}\right)<2,\end{equation*}
hence, if $\pf(\omega)/\pF(\omega)\geq 1$ then by \eqref{eq:ln2}, \eqref{denum} and \eqref{p1a4},
\begin{multline}\label{-1}
\frac{\pf(\omega)}{\pF(\omega)}-1\leq 2\left(\frac{\alpha}{2^m}+4\binom{q}{3}\frac{1}{2^{2n}}+2\binom{q}{2}\frac{1}{2^{2n}}\right)=\\
=2\frac{\alpha}{2^m}+\frac{4}{2^{n-m}}\binom{q}{2}\frac{1}{2^{n+m}}+4\cdot 2\binom{q}{3}\frac{1}{2^{2(n-m)}}\cdot\frac{1}{2^{2m}}\leq 2\frac{\alpha}{2^m}+2\binom{q}{2}\frac{1}{2^{n+m}}+4\frac{\beta}{2^{2m}}.
\end{multline}
The proposition now follows from \eqref{1-} and \eqref{-1}. 
\end{proof}

\subsection{Derivation of Theorem \ref{friendly_theorem}}\label{sec:combine}
We start by bounding $\pF(\bar{S})$ from above.
\begin{lemma}\label{chebyshev}
\begin{equation*}
\pF (\bar{S})\leq\binom{q}{2}\frac{1}{2^{n-m}}\cdot\frac{1}{{\alpha}^2}+\binom{q}{3}\frac{1}{2^{2(n-m)}}\cdot\frac{1}{\beta}\leq\frac{1}{2}\left(\frac{q}{2^{\frac{n+m}{2}}}\right)^2\left(\frac{2^m}{\alpha}\right)^2+\frac{1}{6\cdot 2^{\frac{n-3m}{2}}}\left(\frac{q}{2^{\frac{n+m}{2}}}\right)^3\frac{2^{2m}}{\beta}.
\end{equation*}
\end{lemma}

\begin{proof}
By Chebyshev inequality,
$$\pF\left(\left\{\omega\in\Omega:\left\lvert \col_2(\omega)-\binom{q}{2}\frac{1}{2^{n-m}}\right\rvert>\alpha \right\}\right)\leq\frac{\Var_{{f}} \col_2}{{\alpha}^2},$$
and by Markov inequality,
$$\pF\left(\left\{\omega\in\Omega: \col_{3}(\omega)>\beta\right\}\right)\leq\frac{\E_{{f}} \col_{3}}{\beta}.$$
Using the union bound, we conclude that
$$\pF(\bar{S})\leq\frac{\Var_{{f}} \col_2}{{\alpha}^2}+\frac{\E_{{f}} \col_3}{\beta},$$
and the claim follows by Lemma \ref{collisions}.
\end{proof}

We now combine Lemma~\ref{advantage}, Lemma~\ref{chebyshev} and Proposition~\ref{proposition}.
\begin{lemma}
\label{S to Adv}
Suppose that $q\leq 2^{n-1}$,
\begin{align*} 
&\frac{\alpha}{2^m}+\frac{2}{3}\cdot\frac{1}{2^{\frac{n-3m}{2}}}\left(\frac{q}{2^{\frac{n+m}{2}}}\right)^3\leq\frac{1}{2},\\
&\left(\frac{q}{2^{\frac{n+m}{2}}}\right)^2+\frac{2\alpha}{2^{2m}}\leq\frac{1}{2}\left( \frac{1}{2}-\frac{1}{2^m}\right),\\ 
&\beta\geq 2\binom{q}{3}\frac{1}{2^{2(n-m)}}.
\end{align*} 
Then
\begin{multline}\label{lem7}Adv\leq \left(\frac{q}{2^{\frac{n+m}{2}}}\right)^2+\left(2\frac{\alpha}{2^m}+\frac{1}{2}\left(\frac{q}{2^{\frac{n+m}{2}}}\right)^2\left(\frac{2^m}{\alpha}\right)^2\right)+\\+\left(4\frac{\beta}{2^{2m}}+\frac{1}{6\cdot 2^{\frac{n-3m}{2}}}\left(\frac{q}{2^{\frac{n+m}{2}}}\right)^3\frac{2^{2m}}{\beta}\right).\end{multline}
\end{lemma}
\begin{proof} All the conditions of Proposition \ref{proposition} are clearly satisfied. Therefore,
\begin{gather*}\max_{\omega\in S}\left\lvert\frac{\pf(\omega)}{\pF(\omega)}-1\right\rvert\leq  2\frac{\alpha}{2^m}+2\binom{q}{2}\frac{1}{2^{n+m}}+4\frac{\beta}{2^{2m}}
\leq 2\frac{\alpha}{2^m}+\left(\frac{q}{2^{\frac{n+m}{2}}}\right)^2+4\frac{\beta}{2^{2m}},\end{gather*}
and the claim follows by Lemma \ref{advantage} and Lemma \ref{chebyshev}.
\end{proof}
Theorem \ref{friendly_theorem} now follows by taking  $\alpha,\beta$ to minimize the right hand side of~\eqref{lem7}.
\begin{proof}[Proof of Theorem \ref{friendly_theorem}] 
If $q\geq\frac{1}{4}2^{\frac{m+n}{2}}$ then \eqref{thm1} holds since surely $Adv\leq 1$ and 
\begin{equation*}1=2\sqrt[3]{2}\left(\frac{1}{4}\right)^{2/3}\leq 2\sqrt[3]{2}\left(\frac{q}{2^{\frac{n+m}{2}}}\right)^{2/3}<2\sqrt[3]{2}\left(\frac{q}{2^{\frac{n+m}{2}}}\right)^{2/3}+\frac{2\sqrt{2}}{\sqrt{3}}\left(\frac{q}{2^{\frac{n+m}{2}}}\right)^{3/2}+\left(\frac{q}{2^{\frac{n+m}{2}}}\right)^2.
\end{equation*}
If $m\leq 5$ and $q<\frac{1}{4}2^{\frac{m+n}{2}}$, then \eqref{thm1} also holds, since by the "Birthday" bound \eqref{birthday}, which is obviously still valid when the adversary gets only truncated replies from the oracle,
\begin{multline*}
Adv<\frac{q^2}{2^{n+1}}=2^{m-5}\left(\frac{4q}{2^{\frac{m+n}{2}}}\right)^2<\left(\frac{4q}{2^{\frac{m+n}{2}}}\right)^{2/3}=2\sqrt[3]{2}\left(\frac{q}{2^{\frac{n+m}{2}}}\right)^{2/3}<\\
<2\sqrt[3]{2}\left(\frac{q}{2^{\frac{n+m}{2}}}\right)^{2/3}+\frac{2\sqrt{2}}{\sqrt{3}}\left(\frac{q}{2^{\frac{n+m}{2}}}\right)^{3/2}+\left(\frac{q}{2^{\frac{n+m}{2}}}\right)^2.
\end{multline*}
Finally, if $6\leq m\leq\frac{n}{3}$ and $q<\frac{1}{4}2^{\frac{m+n}{2}}$, then it is straightforward to verify that all the conditions of Lemma \ref{S to Adv} are satisfied if we choose
$$\alpha:=\frac{2^m}{\sqrt[3]{4}}\left(\frac{q}{2^{\frac{n+m}{2}}}\right)^{2/3},\quad\beta:=\frac{2^{2m}}{2\sqrt{6}\cdot 2^{\frac{n-3m}{4}}}\left(\frac{q}{2^{\frac{n+m}{2}}}\right)^{3/2}.$$
Then, by Lemma \ref{S to Adv},
\begin{align*}
Adv\leq&\left(\frac{q}{2^{\frac{n+m}{2}}}\right)^2+\left(2\frac{\alpha}{2^m}+\frac{1}{2}\left(\frac{q}{2^{\frac{n+m}{2}}}\right)^2\left(\frac{2^m}{\alpha}\right)^2\right)+\left(4\frac{\beta}{2^{2m}}+\frac{1}{6\cdot 2^{\frac{n-3m}{2}}}\left(\frac{q}{2^{\frac{n+m}{2}}}\right)^3\frac{2^{2m}}{\beta}\right)=\\
=&\left(\frac{q}{2^{\frac{n+m}{2}}}\right)^2+2\sqrt[3]{2}\left(\frac{q}{2^{\frac{n+m}{2}}}\right)^{2/3}+\frac{2\sqrt{2}}{\sqrt{3}\cdot 2^{\frac{n-3m}{4}}}\left(\frac{q}{2^{\frac{n+m}{2}}}\right)^{3/2}\leq\\
\leq& 2\sqrt[3]{2}\left(\frac{q}{2^{\frac{n+m}{2}}}\right)^{2/3}+\frac{2\sqrt{2}}{\sqrt{3}}\left(\frac{q}{2^{\frac{n+m}{2}}}\right)^{3/2}+\left(\frac{q}{2^{\frac{n+m}{2}}}\right)^2.
\qedhere\end{align*}
\end{proof}

\section{Proof of Theorem \ref{friendly_theorem2}}
\label{sec:main_theorem2}

The proof of Theorem \ref{friendly_theorem2} is more elaborate and technical than the proof of Theorem \ref{friendly_theorem}, but goes along a similar path. It uses statements that are analogous to those used in the proof of Theorem \ref{friendly_theorem}.
As mentioned in Section \ref{sec3}, we let
\begin{equation*}S=\left\{\omega\in\Omega\mid \forall 2\leq j\leq t:\left\lvert \col_j(\omega)-\binom{q}{j}\frac{1}{2^{(j-1)(n-m)}}\right\rvert\leq\alpha_j,\,\col_{t+1}(\omega)\leq\beta\right\},\end{equation*}
where $t\geq 2$ is an integer and $\alpha_1,\alpha_2,\ldots,\alpha_{t-1},\beta$ are positive real numbers to be specifired later.

\subsection{Bounding $\lvert \pf/\pF-1 \rvert$ in $S$}
\label{subsec:technical2}

The following proposition is analogous to Proposition \ref{proposition}.

\begin{proposition}\label{proposition2}
Suppose that $m\leq n-2$, $q\leq 2^{n-1}$, $2\leq t\leq 2^{(m-1)/2}+1$, 
\begin{align}
&\binom{q}{2}\frac{1}{2^{n-m}}+\alpha_1\leq\binom{2^{m-1}}{2},\label{eq:assumption_2_1}\\
&4\left(\frac{q}{2^n}\right)^2+\frac{1}{2t(t+1)2^{\frac{n-m}{2}(t-\frac{n+m}{n+m})}}\left(\frac{2(q+t-1)}{2^{\frac{n+m}{2}}}\right)^{t+1}+\sum_{j=1}^{t-1}\frac{(j-1)!}{2^{jm}}\alpha_j\leq\frac{1}{2},\label{eq:assumption_2_2}\\
&\beta\geq 2\binom{q}{t+1}\frac{1}{2^{t(n-m)}}.\label{eq:assumption_2_3}
\end{align} 
Then for every $\omega\in S$,
\begin{multline*}\left\lvert\frac{\pf(\omega)}{\pF(\omega)}-1\right\rvert\leq 4\left(1+\left(\frac{2(t-1)}{2^m}+\frac{2q}{2^n}\right)^{t-2}\right)\left(\frac{q}{2^{\frac{n+m}{2}}}\right)^2+\\
+2\left(1+\frac{4}{2^m}\right)\frac{\alpha_1}{2^m}+\left(\sum_{j=2}^{t-1}\left(1+2^j\right)\frac{(j-1)!}{2^{jm}}\alpha_j\right)+\frac{2^t(t -1)!}{2^{tm}}\beta.
\end{multline*}
\end{proposition}
In the proof of the proposition we use several lemmas.
\begin{lemma}\label{short}
For every integers $i\geq 0$, $j\geq 1$,
$$0\leq\frac{ i^j}{j!}-\binom{i}{j}\leq\sum_{r=1}^{j-1}\binom{j-1}{r-1}\binom{i}{r}.$$
\end{lemma}
\begin{proof} Note that
\begin{equation*} i(i-1)\ldots(i-(j-1))\leq i^j\leq (i+j-1)\ldots(i+1)i,$$ 
hence
$$\binom{i}{j}\leq\frac{i^j}{j!}\leq\binom{i+j-1}{j}=\sum_{r=1}^j\binom{j-1}{j-r}\binom{i}{r},\end{equation*}
and the claim follows.
\end{proof}
With this we prove the following lemma.
\begin{lemma}\label{lemma ikt2}
Let $ i,k,t$ be integers such that $k\geq 1$, $0\leq i\leq 2^{k-1}$, $t\geq 2$. Then,
\begin{multline}\label{eq:ikt2}0\leq-\ln\left(1-\frac{i}{2^k}\right)-\sum_{j=1}^{t-1}\frac{(j-1)!}{2^{jk}}\binom{i}{j}\leq\\
\leq\sum_{j=1}^{t-1}\frac{(j-1)!}{2^{jk}}\sum_{r=1}^{j-1}\binom{j-1}{r-1}\binom{i}{r}+\frac{2^t(t-1)!}{2^{tk}}\sum_{r=1}^{t}\binom{t-1}{r-1}\binom{i}{r}.
\end{multline}
\end{lemma}
\begin{proof}
By Taylor expansion, for every $x>0$ 
$$-\ln(1-x)=\sum_{j=1}^{t-1}\frac{x^j}{j}+\frac{x^t}{t(1-\xi)^t}$$
for some $0<\xi<x$. It follows that for every $0\leq x\leq\frac{1}{2}$,
$$0\leq-\ln(1-x)-\sum_{j=1}^{t-1}\frac{x^j}{j}\leq\frac{{2}^t}{t}x^{t}.$$
In particular,
\begin{equation}\label{4.4}0\leq-\ln\left(1-\frac{i}{2^k}\right)-\sum_{j=1}^{t-1}\frac{1}{j}\left(\frac{i}{2^k}\right)^j\leq\frac{2^t}{t}\left(\frac{i}{2^k}\right)^t.\end{equation}
and by Lemma \ref{short},
$$0\leq\sum_{j=1}^{t-1}\frac{(j-1)!}{2^{jk}}\left(\frac{i^j}{j!}-\binom{i}{j}\right)\leq\sum_{j=1}^{t-1}\frac{(j-1)!}{2^{jk}}\sum_{r=1}^{j-1}\binom{j-1}{r-1}\binom{i}{r},$$
hence
\begin{equation*}0\leq-\ln\left(1-\frac{i}{2^k}\right)-\sum_{j=1}^{t-1}\frac{(j-1)!}{2^{jk}}\binom{i}{j}\leq\sum_{j=1}^{t-1}\frac{(j-1)!}{2^{jk}}\sum_{r=1}^{j-1}\binom{j-1}{r-1}\binom{i}{r}+\frac{2^t}{t}\left(\frac{i}{2^k}\right)^t,
\end{equation*}
and the lemma follows since by Lemma \ref{short},
\begin{equation*}\frac{2^t}{t}\left(\frac{i}{2^k}\right)^t=\frac{2^t(t-1)!}{2^{tk}}\cdot\frac{i^t}{t!}\leq\frac{2^t(t-1)!}{2^{tk}}\sum_{r=1}^{t}\binom{t-1}{r-1}\binom{i}{r}.\qedhere\end{equation*}
\end{proof}
This leads to the following lemma, which is a generalization of Lemma \ref{lemma skt}.
\begin{lemma}\label{lemma skt2}
Let $s,k,t$ be positive integers such that $s\leq 2^{k-1}$, $2\leq t\leq 2^{k/2}+2$. Then,
\begin{equation*}
0\leq -\ln\prod_{i=0}^{s-1}\left(1-\frac{i}{2^k}\right)-\sum_{j=1}^{t-1}\frac{(j-1)!}{2^{jk}}\binom{s}{j+1}\leq 8\binom{s}{2}\frac{1}{2^{2k}}+
\frac{2^t(t -1)!}{2^{tk}}\sum_{j=1}^{t}\binom{t-1}{j-1}\binom{s}{j+1}.
\end{equation*}
\end{lemma}
\begin{proof}
By summing up \eqref{eq:ikt2} for $0\leq i\leq s-1$ and using the identity
$ \sum_{i=0}^{s-1}\binom{i}{j}=\binom{s}{j+1}$ we get that
\begin{multline*}0\leq-\ln\prod_{i=0}^{s-1}\left(1-\frac{i}{2^k}\right)-\sum_{j=1}^{t-1}\frac{(j-1)!}{2^{jk}}\binom{s}{j+1}\leq\\
\leq\sum_{j=1}^{t-1}\frac{(j-1)!}{2^{jk}}\sum_{r=1}^{j-1}\binom{j-1}{r-1}\binom{s}{r+1}+\frac{2^t(t -1)!}{2^{tk}}\sum_{r=1}^{t}\binom{t-1}{r-1}\binom{s}{r+1}.
\end{multline*}
For every $r,j$ such that $2\leq r+1\leq  j< t-1$,
$$\frac{j!}{2^{(j+1)k}}\binom{j}{r-1} \bigg/ \frac{(j-1)!}{2^{jk}}\binom{j-1}{r-1}=\frac{j}{2^k}\cdot\frac{j}{j-(r-1)}\leq\frac{(t-2)^2}{2^k\cdot 2}\leq\frac{1}{2},$$
hence, for every $1\leq r\leq t-2$,
$$\sum_{j=r+1}^{t-1}\frac{(j-1)!}{2^{jk}}\binom{j-1}{r-1}<2\cdot\frac{r!\cdot r}{2^{rk}}.$$
Therefore,
\begin{multline*}\sum_{j=1}^{t-1}\frac{(j-1)!}{2^{jk}}\sum_{r=1}^{j-1}\binom{j-1}{r-1}\binom{s}{r+1}=\sum_{r=1}^{t-2}\left(\sum_{j=r+1}^{t-1}\frac{(j-1)!}{2^{jk}}\binom{j-1}{r-1}\right)\binom{s}{r+1}<\\
<\sum_{r=1}^{t-2}2\frac{r!\cdot r}{2^{(r+1)k}}\binom{s}{r+1}\leq\sum_{r=1}^{t-2}2\frac{s(s-1)^r}{2^{(r+1)k}}=4\binom{s}{2}\frac{1}{2^{2k}}\sum_{r=1}^{t-2}\left(\frac{s-1}{2^k}\right)^{r-1}<8\binom{s}{2}\frac{1}{2^{2k}},
\end{multline*}
and the lemma follows.
\end{proof}
This leads to the following lemma, which is a generalization of Lemma \ref{ln}.
\begin{lemma}\label{ln2}
If $q\leq 2^{n-1}$, $2\leq t\leq 2^{m/2}+2$, then for every $\omega\in\Omega$ for which $\col_2(\omega)\leq\binom{2^{m-1}}{2}$,
\begin{multline*}0\leq-\ln\frac{\pf(\omega)}{\pF(\omega)}-\ln\prod_{i=0}^{q-1}\left(1-\frac{i}{2^n}\right)-\sum_{j=1}^{t-1}\frac{(j-1)!}{2^{jm}}\col_{j+1}(\omega)\leq\\
\leq 8\frac{\col_2(\omega)}{2^{2m}}+\frac{2^t(t -1)!}{2^{tm}}\sum_{j=1}^{t}\binom{t-1}{j-1}\col_{j+1}(\omega).\end{multline*}
\end{lemma}
\begin{proof}
Suppose that in the q-tuple $\omega$ exactly $t$ distinct vectors in $\{0,1\}^{n-m}$
appear, with multiplicities $d_1,d_2,\ldots,d_l$ respectively.
For every $1\leq k\leq l$, note that $d_k\leq 2^{m-1}$, since $\sum_{k=1}^l\binom{d_k}{2}=\col_2(\omega)\leq\binom{2^{m-1}}{2}$, hence by Lemma \ref{lemma skt2}, 
\begin{equation*}0\leq-\ln\prod_{i=0}^{d_k-1}\left(1-\frac{i}{2^m}\right)-\sum_{j=1}^{t-1}\frac{(j-1)!}{2^{jm}}\binom{d_k}{j+1}\leq 8\binom{d_k}{2}\frac{1}{2^{2m}}+\frac{2^t(t -1)!}{2^{tm}}\sum_{j=1}^{t}\binom{t-1}{j-1}\binom{d_k}{j+1}.
\end{equation*}
Summing up on  $1\leq k\leq l$, we get
\begin{multline*}0\leq-\sum_{k=1}^l\ln\prod_{i=0}^{d_k-1}\left(1-\frac{i}{2^m}\right)-\sum_{j=1}^{t-1}\frac{(j-1)!}{2^{jm}}\sum_{k=1}^l\binom{d_k}{j+1}\leq\\
\leq\frac{8}{2^{2m}}\sum_{k=1}^l\binom{d_k}{2}+\frac{2^t(t -1)!}{2^{tm}}\sum_{j=1}^{t}\binom{t-1}{j-1}\sum_{k=1}^l\binom{d_k}{j+1},
\end{multline*}
and the claim follows by \eqref{eq:pr_pr}.
\end{proof}
Finally, we also need the following technical lemma. 
\begin{lemma}\label{tech}For every integers $n,m,q$ and $t$
$$\frac{2^t(t-1)!}{2^{tm}}\sum_{j=1}^{t-1}\binom{t-1}{j-1}\binom{q}{j+1}\frac{1}{2^{j(n-m)}}\leq 4\left(\frac{q}{2^{\frac{n+m}{2}}}\right)^2\left(\frac{2(t-1)}{2^m}+\frac{2q}{2^n}\right)^{t-2}.$$
\end{lemma}
\begin{proof} Note that
\begin{align*}(t-1)!&\sum_{j=1}^{t-1}\binom{t-1}{j-1}\binom{q}{j+1}\frac{1}{2^{j(n-m)}}\leq(t-1)!\sum_{j=1}^{t-1} (t-1)\binom{t-2}{j-1}\frac{q^{j+1}}{(j+1)!}\frac{1}{2^{j(n-m)}}=\\
&=\frac{q^2}{2^{n-m}}\sum_{j=1}^{t-1}\binom{t-2}{j-1}\frac{(t-1)(t-1)!}{(j+1)!}\left(\frac{q}{2^{n-m}}\right)^{j-1}\leq\\
&\leq\frac{q^2}{2^{n-m}}\sum_{j=1}^{t-1}\binom{t-2}{j-1}(t-1)^{t-j-1}\left(\frac{q}{2^{n-m}}\right)^{j-1}=\frac{q^2}{2^{n-m}}\left((t-1)+\frac{q}{2^{n-m}}\right)^{t-2},
\end{align*}
hence
\begin{multline*}\frac{2^t(t-1)!}{2^{tm}}\sum_{j=1}^{t-1}\binom{t-1}{j-1}\binom{q}{j+1}\frac{1}{2^{j(n-m)}}\leq\frac{2^t}{2^{tm}}\cdot\frac{q^2}{2^{n-m}}\left((t-1)+\frac{q}{2^{n-m}}\right)^{t-2}=\\
=4\left(\frac{q}{2^{\frac{n+m}{2}}}\right)^2\left(\frac{2(t-1)}{2^m}+\frac{2q}{2^n}\right)^{t-2}.\qedhere\end{multline*}
\end{proof}
\begin{proof}[Proof of Proposition~\ref{proposition2}]
For every $\omega\in S$, by \eqref{eq:assumption_2_1},
$$\col_2(\omega)\leq\binom{q}{2}\frac{1}{2^{n-m}}+\alpha_1\leq\binom{2^{m-1}}{2},$$ 
hence by Lemma \ref{ln2} and Lemma \ref{lemma skt2},
\begin{align}
-\ln\frac{\pf(\omega)}{\pF(\omega)}\leq 8\frac{\col_2(\omega)}{2^{2m}}+\frac{2^t(t -1)!}{2^{tm}}\sum_{j=1}^{t}&\binom{t-1}{j-1}\col_{j+1}(\omega)+\nonumber\\
&+\sum_{j=1}^{t-1}\frac{(j-1)!}{2^{jm}}\left(\col_{j+1}(\omega)-\binom{q}{j+1}\frac{1}{2^{j(n-m)}}\right),\label{-2}\\
\ln\frac{\pf(\omega)}{\pF(\omega)}\leq  8\binom{q}{2}\frac{1}{2^{2n}}+\frac{2^t(t-1)!}{2^{tn}}\sum_{j=1}^{t}&\binom{t-1}{j-1}\binom{q}{j+1}-\nonumber\\
&-\sum_{j=1}^{t-1}\frac{(j-1)!}{2^{jm}}\left(\col_{j+1}(\omega)-\binom{q}{j+1}\frac{1}{2^{j(n-m)}}\right).\label{+2}
\end{align}
By \eqref{eq:ln1}, \eqref{-2} and the definition of $S$,
\begin{align}1-\frac{\pf(\omega)}{\pF(\omega)}\leq&-\ln\frac{\pf(\omega)}{\pF(\omega)}\leq\nonumber\\
\leq&\frac{8}{2^{2m}}\left(\binom{q}{2}\frac{1}{2^{n-m}}+\alpha_1\right)+\nonumber\\
&+\frac{2^t(t -1)!}{2^{tm}}\left(\sum_{j=1}^{t-1}\binom{t-1}{j-1}\left(\binom{q}{j+1}\frac{1}{2^{j(n-m)}}+\alpha_j\right)+\beta\right)+\sum_{j=1}^{t-1}\frac{(j-1)!}{2^{jm}}\alpha_j=\nonumber\\
=&8\binom{q}{2}\frac{1}{2^{n+m}}+\frac{2^t(t-1)!}{2^{tm}}\sum_{j=1}^{t-1}\binom{t-1}{j-1}\binom{q}{j+1}\frac{1}{2^{j(n-m)}}+\nonumber\\
&+8\frac{\alpha_1}{2^{2m}}+\left(\sum_{j=1}^{t-1}\left(\frac{(j-1)!}{2^{jm}}+\frac{2^t(t-1)!}{2^{tm}}\binom{t-1}{j-1}\right)\alpha_j\right)+\frac{2^t(t-1)!}{2^{tm}}\beta\leq\nonumber\\
\leq& 4\left(\frac{q}{2^{\frac{n+m}{2}}}\right)^2+4\left(\frac{q}{2^{\frac{n+m}{2}}}\right)^2\left(\frac{2(t-1)}{2^m}+\frac{2q}{2^n}\right)^{t-2}+\nonumber\\
&+\left(2+\frac{8}{2^m}\right)\frac{\alpha_1}{2^m}+\left(\sum_{j=2}^{t-1}\left(1+2^j\right)\frac{(j-1)!}{2^{jm}}\alpha_j\right)+\frac{2^t(t -1)!}{2^{tm}}\beta,\label{upper2}
\end{align}
where on the last step we used Lemma \ref{tech} and the fact that for every $1\leq j\leq t-1$, since $(t-1)^2\leq 2^{m-1}$,
\begin{multline*}\frac{2^t(t-1)!}{2^{tm}}\binom{t-1}{j-1}\leq\frac{2^t(t-1)!}{2^{tm}}\cdot\frac{(t-1)!}{(j-1)!}=\frac{2^t(j-1)!}{2^{tm}}\left(\frac{(t-1)!}{(j-1)!}\right)^2\leq\\
\leq\frac{2^t(j-1)!}{2^{tm}}(t-1)^{2(t-j)}=2^j\frac{(j-1)!}{2^{jm}}\left(\frac{(t-1)^2}{2^{m-1}}\right)^{t-j}\leq 2^j\frac{(j-1)!}{2^{jm}}.
\end{multline*}
On the other hand, by \eqref{+2} and the definition of $S$,
\begin{equation}\label{eq:upper_ln2}\ln\frac{\pf(\omega)}{\pF(\omega)}\leq 8\binom{q}{2}\frac{1}{2^{2n}}+\frac{2^t(t-1)!}{2^{tn}}\sum_{j=1}^{t}\binom{t-1}{j-1}\binom{q}{j+1}+\sum_{j=1}^{t-1}\frac{(j-1)!}{2^{jm}}\alpha_j.
\end{equation}
In particular, by \eqref{eq:assumption_2_2}
\begin{multline*}\ln\frac{\pf(\omega)}{\pF(\omega)}\leq 8\binom{q}{2}\frac{1}{2^{2n}}+\frac{2^t(t-1)!}{2^{tn}}\binom{q+t-1}{t+1}+\sum_{j=1}^{t-1}\frac{(j-1)!}{2^{jm}}\alpha_j\leq\\
\leq 4\left(\frac{q}{2^n}\right)^2+\frac{1}{2t(t+1)2^{\frac{n-m}{2}(t-\frac{n+m}{n+m})}}\left(\frac{2(q+t-1)}{2^{\frac{n+m}{2}}}\right)^{t+1}+\sum_{j=1}^{t-1}\frac{(j-1)!}{2^{jm}}\alpha_j\leq\frac{1}{2},
\end{multline*}
hence $\pf(\omega)/\pF(\omega)\leq\sqrt{e}<2$. Therefore, if $\pf(\omega)/\pF(\omega)\geq 1$ then by \eqref{eq:ln2}, \eqref{eq:upper_ln2} and \eqref{eq:assumption_2_3},
\begin{align*}\frac{\pf(\omega)}{\pF(\omega)}-1\leq& 2\left(8\binom{q}{2}\frac{1}{2^{2n}}+\frac{2^t(t-1)!}{2^{tn}}\sum_{j=1}^{t}\binom{t-1}{j-1}\binom{q}{j+1}+\sum_{j=1}^{t-1}\frac{(j-1)!}{2^{jm}}\alpha_j\right)\leq\\
\leq& 16\frac{q^2}{2^{2n}}+2\frac{2^t(t-1)!}{2^{tn}}\sum_{j=1}^{t-1}\binom{t-1}{j-1}\binom{q}{j+1}2^{(t-j-1)(n-m)}+\\
&+2\frac{2^t(t-1)!}{2^{tn}}2^{t(n-m)}\beta+2\sum_{j=1}^{t-1}\frac{(j-1)!}{2^{jm}}\alpha_j=\\
=&\frac{4}{2^{n-m-2}}\left(\frac{q}{2^{\frac{n+m}{2}}}\right)^2+\frac{2}{2^{n-m}}\left(\frac{2^t(t-1)!}{2^{tm}}\sum_{j=1}^{t-1}\binom{t-1}{j-1}\binom{q}{j+1}\frac{1}{2^{j(n-m)}}\right)+\\
&+\sum_{j=1}^{t-1}2\frac{(j-1)!}{2^{jm}}\alpha_j+\frac{2^t(t -1)!}{2^{tm}}\beta,
\end{align*}
and we are done by Lemma \ref{tech}.
\end{proof}

\subsection{Derivation of Theorem \ref{friendly_theorem2}}
\label{subsec:combine}
The following lemma is a generalization of Lemma \ref{chebyshev}.
\begin{lemma}
\label{chebyshev2}
\begin{equation*}\pF (\bar{S})\leq\sum_{j=1}^{t-1}\binom{j+1}{2}\binom{q}{j+1}\frac{1}{2^{j(n-m)}}\left(1+\frac{q}{2^{n-m}}\right)^{j-1}\frac{1}{{\alpha_j}^2}+\binom{q}{t+1}\frac{1}{2^{t(n-m)}\beta}.
\end{equation*}
\end{lemma}
\begin{proof}
By Chebyshev inequality, for every $1\leq j\leq  t-1$,
$$\pF\left(\left\{\omega\in\Omega:\left\lvert \col_{j+1}(\omega)-\binom{q}{j+1}\frac{1}{2^{j(n-m)}}\right\rvert>\alpha_j \right\}\right)\leq\frac{\Var_{f} \col_{j+1}}{{\alpha_j}^2},$$
and by Markov inequality,
$$\pF\left(\left\{\omega\in\Omega: \col_{t+1}(\omega)>\beta\right\}\right)\leq\frac{\E_{f} \col_{t+1}}{\beta}.$$
Using the union bound, we conclude that
$$\pF(\bar{S})\leq\left(\sum_{j=1}^{t-1}\frac{\Var_{f} \col_{j+1}}{{\alpha_j}^2}\right)+\frac{\E_{f} \col_{t+1}}{\beta},$$
and the claim follows by Lemma \ref{collisions}.
\end{proof}

Combining Lemma~\ref{advantage}, Lemma~\ref{chebyshev2} and Proposition~\ref{proposition2}, we get the following generalization of Lemma \ref{S to Adv}.

\begin{lemma}
\label{S to Adv2}
Suppose that $m\leq n-2$, $q\leq 2^{n-1}$, $2\leq t\leq 2^{(m-1)/2}+1$,
\begin{align*}
&\binom{q}{2}\frac{1}{2^{n-m}}+\alpha_1\leq\binom{2^{m-1}}{2},\\
&4\left(\frac{q}{2^n}\right)^2+\frac{1}{2t(t+1)2^{\frac{n-m}{2}(t-\frac{n+m}{n+m})}}\left(\frac{2(q+t-1)}{2^{\frac{n+m}{2}}}\right)^{t+1}+\sum_{j=1}^{t-1}\frac{(j-1)!}{2^{jm}}\alpha_j\leq\frac{1}{2},\\
&\beta\geq 2\binom{q}{t+1}\frac{1}{2^{t(n-m)}}.
\end{align*}
Then
\begin{multline}\label{lem14}Adv\leq 4\left(1+\left(\frac{2(t-1)}{2^m}+\frac{2q}{2^n}\right)^{t-2}\right)\left(\frac{q}{2^{\frac{n+m}{2}}}\right)^2+\left(2\left(1+\frac{4}{2^m}\right)\frac{\alpha_1}{2^m}+\frac{1}{2}\left(\frac{q}{2^{\frac{n+m}{2}}}\right)^2\left(\frac{2^m}{\alpha_1}\right)^2\right)+\\
+\sum_{j=2}^{t-1}\left(\left(1+2^j\right)\frac{(j-1)!\alpha_j}{2^{jm}}+\frac{(j-1)!}{2\cdot 2^{(j-1)(n-m)}}\left(\frac{q}{2^{\frac{n+m}{2}}}\right)^{j+1}\left(\frac{1}{2^{\frac{3m-n}{2}}}+\frac{q}{2^{\frac{n+m}{2}}}\right)^{j-1}\left(\frac{2^{jm}}{(j-1)!\alpha_j}\right)^2\right)+\\
+\left(\frac{2^t(t -1)!}{2^{tm}}\beta+\frac{q^{t+1}}{(t+1)!}\cdot\frac{1}{2^{t(n-m)}\beta}\right).\end{multline}
\end{lemma}
Taking  $\alpha_1,\alpha_2,\ldots,\alpha_{t-1},\beta$ to minimize the right hand side of \eqref{lem14} we get the following.

\begin{lemma}
\label{lemma92}
Assume that $m\leq n-2$, $q\leq 2^{n-1}$, $2\leq t\leq 2^{(m-1)/2}+1$,
\begin{align*}
&\frac{1}{2^{\frac{n-m}{2}(t-\frac{n+m}{n-m})}}\left(\frac{2q}{2^{\frac{n+m}{2}}}\right)^{t+1}\leq\frac{t(t+1)}{4},\\
&\left(\frac{q}{2^{\frac{n+m}{2}}}\right)^2+\frac{1}{\sqrt[3]{1+\frac{4}{2^m}}}\left(\frac{q}{2^{\frac{n+m}{2}}}\right)^{2/3}\frac{1}{2^m}\leq\frac{1}{2}\left(\frac{1}{2}-\frac{1}{2^m}\right),\\
&\frac{4(t-2)}{2^{n-m}}\left(\frac{1}{2^{\frac{3m-n}{2}}}+\frac{q}{2^{\frac{n+m}{2}}}\right)\frac{q}{2^{\frac{n+m}{2}}}\leq\frac{1}{8},
\end{align*}
and
\begin{multline*}
4\left(\frac{q}{2^n}\right)^2+\frac{1}{2t(t+1)2^{\frac{n-m}{2}(t-\frac{n+m}{n+m})}}\left(\frac{2(q+t-1)}{2^{\frac{n+m}{2}}}\right)^{t+1}+\\
+\frac{1}{2\sqrt[3]{1+\frac{4}{2^m}}}\left(\frac{q}{2^{\frac{n+m}{2}}}\right)^{2/3}+\left(\frac{2}{5\cdot 2^{n-m}}\left(\frac{1}{2^{\frac{3m-n}{2}}}+\frac{q}{2^{\frac{n+m}{2}}}\right)\right)^{\frac{1}{3}}\frac{q}{2^{\frac{n+m}{2}}}\leq\frac{1}{2}.
\end{multline*}
Then
\begin{multline}\label{eq:adv2_final}
Adv\leq  4\left(1+\left(\frac{2(t-1)}{2^m}+\frac{2q}{2^n}\right)^{t-2}\right)\left(\frac{q}{2^{\frac{n+m}{2}}}\right)^2
+2\sqrt[3]{2}\left(1+\frac{4}{2^m}\right)^{2/3}\left(\frac{q}{2^{\frac{n+m}{2}}}\right)^{2/3}+\\
+2\sqrt[3]{100}\left(\frac{1}{2^{n-m}}\left(\frac{1}{2^{\frac{3m-n}{2}}}+\frac{q}{2^{\frac{n+m}{2}}}\right)\right)^{\frac{1}{3}}\frac{q}{2^{\frac{n+m}{2}}}+\frac{\sqrt{2}}{\sqrt{t(t+1)}}\cdot\frac{1}{2^{\frac{n-m}{4}\left(t-\frac{n+m}{n-m}\right)}}\left(\frac{2q}{2^{\frac{n+m}{2}}}\right)^{\frac{t+1}{2}}.
\end{multline}
\end{lemma}
\begin{proof}
Take
\begin{align*}\alpha_1&:=\frac{2^m}{\sqrt[3]{4\left(1+\frac{4}{2^m}\right)}}\left(\frac{q}{2^{\frac{n+m}{2}}}\right)^{2/3}~,\\
\alpha_j&:=\frac{2^{jm}}{\sqrt[3]{2\left(1+2^j\right)(j-1)!^2}}\left(\frac{1}{2^{(j-1)(n-m)}}\left(\frac{q}{2^{\frac{n+m}{2}}}\right)^{j+1}\left(\frac{1}{2^{\frac{3m-n}{2}}}+\frac{q}{2^{\frac{n+m}{2}}}\right)^{j-1}\right)^{\frac{1}{3}}\end{align*}
for $2\leq j\leq t-1$, and
$$\beta:=\frac{2^{tm}}{(t-1)!}\left(\frac{q^{t+1}}{t(t+1)\cdot 2^{t(n+1)}}\right)^{\frac{1}{2}},$$
and use Lemma \ref{S to Adv2}.
\end{proof}
The proof of Theorem~\ref{friendly_theorem2} follows by taking $t=\lceil\frac{n+m}{n-m}\rceil$ in Lemma \ref{lemma92}.

\begin{proof}[Proof of Theorem \ref{friendly_theorem2}]
The inequality \eqref{thm2} clearly holds for $q\geq\frac{1}{8}2^{\frac{m+n}{2}}$ since surely $Adv\leq 1$ and 
\begin{multline*}
1=3\left(\frac{1}{8}\right)^{2/3}+2\left(\frac{1}{8}\right)\leq 3\left(\frac{q}{2^{\frac{n+m}{2}}}\right)^{2/3}+2\left(\frac{q}{2^{\frac{n+m}{2}}}\right)<\\
<3\left(\frac{q}{2^{\frac{n+m}{2}}}\right)^{2/3}+2\left(\frac{q}{2^{\frac{n+m}{2}}}\right)+5\left(\frac{q}{2^{\frac{n+m}{2}}}\right)^2+\frac{1}{2}\left(\frac{2q}{2^{\frac{n+m}{2}}}\right)^{\frac{n}{n-m}}.
\end{multline*}
We therefore assume $q<\frac{1}{8}2^{\frac{m+n}{2}}$, and let $t:=\lceil\frac{n+m}{n-m}\rceil$. Note that $5\leq m\leq n-8$, since $\frac{n}{3}<m\leq n-\log_2 n-4$.
Therefore,
\begin{align*}
\frac{2(t-1)}{2^m}&\leq\frac{2\cdot\frac{n+m}{n-m}}{2^m}=\frac{\frac{4m}{n-m}+2}{2^m}\leq\frac{m+4}{2^{m+1}}\leq\frac{9}{64},\\
\frac{2q}{2^n}&=\frac{2}{2^{\frac{n-m}{2}}}\cdot\frac{q}{2^{\frac{n+m}{2}}}\leq\frac{1}{64},
\end{align*}
hence
\begin{equation}\label{eq:c1}4\left(1+\left(\frac{2(t-1)}{2^m}+\frac{2q}{2^n}\right)^{t-2}\right)\leq 4\left(1+\left(\frac{9}{64}+\frac{1}{64}\right)\right)<5.\end{equation}
Additionaly, using also that $t\geq 3$, since $m>\frac{n}{3}$,
\begin{align}2\sqrt[3]{2}\left(1+\frac{4}{2^m}\right)^{2/3}&\leq 2\sqrt[3]{2}\left(1+\frac{4}{2^5}\right)^{2/3}<3,\label{eq:c2}\\
2\sqrt[3]{100}\left(\frac{1}{2^{n-m}}\left(\frac{1}{2^{\frac{3m-n}{2}}}+\frac{q}{2^{\frac{n+m}{2}}}\right)\right)^{\frac{1}{3}}&<2\sqrt[3]{100}\left(\frac{1}{2^8}\left(1+\frac{1}{8}\right)\right)^{\frac{1}{3}}<2,\label{eq:c3}\\
\frac{\sqrt{2}}{\sqrt{t(t+1)}}\cdot\frac{1}{2^{\frac{n-m}{4}\left(t-\frac{n+m}{n-m}\right)}}&\leq\frac{\sqrt{2}}{\sqrt{3\cdot 4}}<\frac{1}{2},\label{eq:c4}
\end{align}
and finally, since $2q/2^{\frac{n+m}{2}}<1$ and $\frac{t+1}{2}\geq\frac{n}{n+m}$,
\begin{equation}\label{eq:c5}
\left(\frac{2q}{2^{\frac{n+m}{2}}}\right)^{\frac{t+1}{2}}<\left(\frac{2q}{2^{\frac{n+m}{2}}}\right)^{\frac{n}{n+m}}.
\end{equation}
Since it is straightforward to verify that $n,m,q,t$ satisfy all the conditions of Lemma \ref{lemma92}, we get \eqref{thm2} by combining \eqref{eq:adv2_final}, \eqref{eq:c1}, \eqref{eq:c2}, \eqref{eq:c3}, \eqref{eq:c4} and \eqref{eq:c5}.
\end{proof}

\section{Discussion}

We conclude with the following note. As mentioned above, the analysis in \cite{Hall} is also based on examining the set $S$, but only for the particular choice of parameters: $t=2$, $\alpha_1=c\,q/2^{\frac{n-m+1}{2}}$, $\beta=0$. Choosing  $\alpha_1$ to be proportional to the standard deviation of $\col_2$ seems reasonable and natural (although in our analysis we employ a somewhat different choice). However, choosing  $\beta=0$ is artificial and too restrictive, and limiting $t$ to be $2$ is insufficient for getting the result for large $m$.


\begin{thebibliography}{999999999}

\bibitem{GGM}
S. Gilboa, S. Gueron, B. Morris, How many queries are needed to distinguish a truncated random permutation from a random function?, submitted, preprint available at \texttt{arXiv:1412.5204}.

\bibitem{Hall}
C. Hall, D. Wagner, J. Kelsey, B. Schneier, Building prfs from prps, in: Proceedings of
CRYPTO-98: Advances in Cryptography, Springer Verlag, 1998, pp. 370-389.

\bibitem{Stam}
A. J. Stam, Distance between sampling with and without replacement, Statist. Neerlandica {\bf 32} (1978), no.~2, 81--91. 

\end{thebibliography}
\end{document}